\documentclass[10pt,twoside]{article}
\usepackage{mathrsfs}
\usepackage{amsmath}
\usepackage{amssymb}
\usepackage{fancyhdr}
\usepackage{latexsym}
\usepackage{bbding}
\usepackage{mathrsfs}
\usepackage{wasysym}
\usepackage{multicol,graphics}

\newtheorem{theorem}{Theorem}[section]
\newtheorem{lemma}[theorem]{Lemma}

\numberwithin{equation}{section}
\newenvironment{proof}[1][Proof]{\noindent\textbf{#1.} }{\hfill $\Box$}
\allowdisplaybreaks

 \makeatletter
\begin{document}
\title{{Blow-up Criteria for the Three Dimensional Nonlinear Dissipative System Modeling
Electro-hydrodynamics\footnote{E-mail addresses: jihzhao@163.com (J. Zhao), 124204602@qq.com
(M. Bai).}}}
\author{Jihong Zhao$^{\text{1}}$,
\ \ Meng Bai$^{\text{2}}$\\
[0.2cm] {\small $^{\text{1}}$ College of Science, Northwest A\&F
University,  Yangling, Shaanxi 712100, China}\\
[0.2cm] {\small $^{\text{2}}$  School of Mathematics and Statistics, Zhaoqing University, Zhaoqing, Guangdong 526061, China}}

\date{}
\maketitle

\begin{abstract}
In this paper, we investigate some sufficient conditions for the breakdown of local  smooth solutions to the three dimensional  nonlinear nonlocal dissipative system modeling
electro-hydrodynamics.  This model is a strongly coupled system by the well-known incompressible Navier-Stokes equations and the classical Poisson-Nernst-Planck equations. We show that the maximum of  the vorticity field alone controls the breakdown of smooth solutions, which reveals that the velocity field plays a more
dominant role than the density functions of charged particles in the blow-up theory of the system.
Moreover, some Prodi-Serrin type blow-up criteria are also established.


\end{abstract}

\section{Introduction}

In this paper, we study a well-established mathematical model for electro-diffusion, which is governed by transport, and Lorentz force coupling between the Navier-Stokes equations of an incompressible fluid and the transported Poisson-Nernst-Planck equations of a binary, diffuse charge system (cf. \cite{R90}). The three dimensional Cauchy problem reads as
\begin{equation}\label{eq1.1}
\begin{cases}
  \partial_{t} u+(u\cdot\nabla) u-\mu\Delta
  u+\nabla \Pi=\varepsilon\Delta
  \Psi\nabla\Psi,\ \ & x\in\mathbb{R}^{3},\ t>0,\\
  \nabla\cdot u=0,\ \ & x\in\mathbb{R}^{3},\ t>0,\\
  \partial_{t} v+(u\cdot \nabla)
  v=\nabla\cdot(D_{1}\nabla v-\nu_{1}v\nabla \Psi),\ \ & x\in\mathbb{R}^{3},\ t>0,\\
  \partial_{t} w+(u\cdot \nabla)
  w=\nabla\cdot(D_{2}\nabla w+\nu_{2}w\nabla \Psi),\ \ & x\in\mathbb{R}^{3},\ t>0,\\
  \varepsilon\Delta \Psi=v-w,\ \ & x\in\mathbb{R}^{3},\ t>0,\\
   (u, v, w)|_{t=0}=(u_0, v_0, w_0), \ \ & x\in\mathbb{R}^{3},
\end{cases}
\end{equation}
where $u=(u^{1},u^{2}, u^{3})$ is the velocity vector field,   $\Pi$ is the
scalar pressure,  $\Psi$
is the electrostatic potential,  $v$ and $w$ are the densities of binary diffuse negative and positive charges (e.g., ions),
respectively.  $\mu$ is the kinematic viscosity, $\varepsilon$ is the
dielectric constant of the fluid, known as the Debye length, related to vacuum
permittivity, the relative permittivity and characteristic charge density.  $D_{1}$, $D_{2}$, $\nu_{1}$, $\nu_{2}$ are
the diffusion and mobility coefficients of the charges\footnote{$D_{1}=\frac{kT_{0}\nu_{1}}{e}$,
$D_{2}=\frac{kT_{0}\nu_{2}}{e}$, where $T_{0}$ is the ambient
temperature, $k$ is the Boltzmann constant, and $e$ is the charge
mobility.}.  We refer the reader to see \cite{S09}  for the detailed mathematical description and physical background of this fluid-dynamical model. For the sake of simplicity of presentation, we shall assume that all physical parameters to be one in the system \eqref{eq1.1}.

The system \eqref{eq1.1} was introduced by Rubinstein
in \cite{R90}, which is intended to account for the electrical, fluid-mechanical and bio-chemical phenomena simultaneously occurring in complex bio-hydrid system like bio-chips, bio-reactors or micro-fluidic chambers in Lab-On-Chip technology, see \cite{BW04, EN00, ES05, SK04} for specific  applications of the system
\eqref{eq1.1} and \cite{LCJS081,LCJS082,LMCJS083} for the computational simulations. Generally speaking, the self-consistent charge transport is described by the
Poisson-Nernst-Planck equations, and the fluid motion is governed by the incompressible Navier-Stokes equations with forcing terms.

To our knowledge, mathematical analysis of the system \eqref{eq1.1} was
initiated by  Jerome \cite{J02}, where  the author established a
local well-posedness theory of the system \eqref{eq1.1}
based on the Kato's semigroup framework. Concerning existence results of
weak solutions, we refer the reader to see  \cite{J11, JS09, R09, S09}, where the system \eqref{eq1.1} adapted with various boundary conditions are considered (e.g., Neumann, no-flux, mixed boundary conditions), while considering existence results of smooth solutions, we refer the reader to see \cite{DZC11, ZDC10, ZZL15}, where local existence with any initial data and global existence with small initial data in various functional spaces (e.g., Lebesgue, Besov spaces) are established.  Since the Navier-Stokes equations is a subsystem of \eqref{eq1.1}, one can not expect better results than for the Navier-Stokes equations. Hence, in the case of three dimensional space, whether the regularity and uniqueness of global weak solutions or global existence of smooth solutions are still the challenge open problems. Some regularity and uniqueness issues have been studied by the references  \cite{FG09, FLN13, FNZ12, JF13} even for more general system for the electro-kinetic fluid model.

System \eqref{eq1.1} is closely connected to classical models in fluid mechanics and in the theory research of electrolytes.
When the charge transport is dropped, i.e., $v=w=\Psi=0$, one gets the well-known three dimensional incompressible Navier-Stokes equations:
\begin{equation}\label{eq1.2}
\begin{cases}
  \partial_{t} u+(u\cdot\nabla)u-\mu\Delta
  u+\nabla \Pi=0,\ \ & x\in\mathbb{R}^{3},\ t>0,\\
  \nabla\cdot u=0,\ \ & x\in\mathbb{R}^{3},\ t>0,\\
  u|_{t=0}=u_0,\ \ & x\in\mathbb{R}^{3}.\\
\end{cases}
\end{equation}
This equations has drawn considerable attention of
researchers for many years, and many fruitful results can be found
from the literature, e.g., \cite{BP08,CKN82,FK64,K84,KT01,L34,W15,Y10} and the references
therein.
In the pioneering work \cite{L34}, Leray proved global existence of weak solutions with any initial data in $L^{2}$ in dimensions two and three, moreover, the global existence of smooth  solutions in dimensions two and in dimensions three under smallness condition imposed on the initial data were also addressed. It is well-known that in dimensions three, without smallness condition imposed on initial data,  whether the corresponding local smooth solutions can be extended to the global one is an outstanding open problem (alternative open problem is the regularity and uniqueness of global weak solutions).  The well-known Prodi-Serrin criterion (cf. \cite{P59, S62}) shows
that any weak solution $u$ for the Navier-Stokes equations \eqref{eq1.2} satisfying
\begin{align}\label{eq1.3}
 u\in L^{q}(0,T;L^{p}(\mathbb{R}^{3}))\ \text{ with
  }\ \frac{2}{q}+\frac{3}{p}= 1\ \ \text{ and }\ \ 3<p\leq \infty,
\end{align}
then $u$ is regular on $(0,T]\times \mathbb{R}^{3}$.  Beir\~{a}o da
Veiga \cite{B95} extended the condition \eqref{eq1.3} to the following condition:
\begin{align}\label{eq1.4}
\nabla u\in
  L^{q}(0,T;L^{p}(\mathbb{R}^{3}))\ \text{ with
  }\ \frac{2}{q}+\frac{3}{p}= 2\ \ \text{ and }\ \ \frac{3}{2}<p\leq \infty.
\end{align}
In 1984,
Beale, Kato and Majda in their celebrated work \cite{BKM84} showed that if the smooth solution $u$ blows up at the time
$t=T$, then
\begin{equation}\label{eq1.5}
  \int_{0}^{T}\|\omega(\cdot,
  t)\|_{L^{\infty}}\;dt=\infty,
\end{equation}
where $\omega=\nabla\times u$ is the vorticity. Subsequently, Kozono and Taniuchi \cite{KT00} and Kozono, Ogawa and
Taniuchi \cite{KOT02}, respectively,  refined the criterion \eqref{eq1.5}
to the following two criteria:
\begin{equation}\label{eq1.6}
  \int_{0}^{T}\|\omega(\cdot, t)\|_{BMO}\;dt=\infty\
  \ \text{and}\ \ \int_{0}^{T}\|\omega(\cdot, t)\|_{\dot{B}^{0}_{\infty,
  \infty}}\;dt=\infty,
\end{equation}
where $BMO$ is the space of  \textit{Bounded Mean
Oscillation} and $\dot{B}^{0}_{\infty, \infty}(\mathbb{R}^3)$ is the homogeneous
Besov spaces. Recently, Dong and Zhang \cite{DZ10} improved the criteria \eqref{eq1.6} to the Navier-Stokes equations \eqref{eq1.2} that if
\begin{equation}\label{eq1.7}
  \int_{0}^{T}\|\nabla_{h}u^{h}(\cdot, t)\|_{\dot{B}^{0}_{\infty,
  \infty}}\;dt<\infty,
\end{equation}
then the solution $u$ can be extended beyond the time $T$, where $\nabla_{h}\stackrel{def}{=}(\partial_{1}, \partial_{2})$, and $u^{h}\stackrel{def}{=}(u^{1},u^{2})$ is the horizontal components of the velocity field $u$.

Based on the basic energy inequalities, by using the standard fixed point argument, it is easy to show local well-posedness of the system
\eqref{eq1.1} for sufficiently smooth initial data and hence we have the following result.

\begin{theorem}\label{th1.1}
Let $u_{0}\in H^{3}(\mathbb{R}^{3})$ with
$\nabla\cdot u_{0}=0$,  $v_{0}, w_{0}\in  L^{1}(\mathbb{R}^{3})\cap
H^{2}(\mathbb{R}^{3})$ and $v_{0}, w_{0}\geq0$. Then there exists a time $T_{*}>0$
such that the system \eqref{eq1.1}  admits an unique solution $(u,v,w)$ satisfying
\begin{equation}\label{eq1.8}
\begin{cases}
  u\in C([0,T], H^{3}(\mathbb{R}^{3}))\cap L^{\infty}(0,T; H^{3}(\mathbb{R}^{3}))\cap  L^{2}(0, T;
  H^{4}(\mathbb{R}^{3})),\\
  v,w\in C([0,T], H^{2}(\mathbb{R}^{3}))\cap L^{\infty}(0,T; H^{2}(\mathbb{R}^{3}))\cap  L^{2}(0, T;
  H^{3}(\mathbb{R}^{3}))
\end{cases}
\end{equation}
for any $0<T<T_{*}$. Moreover, $v,w\geq0$ a.e. in  $[0,T_{*})\times \mathbb{R}^{3}$.
\end{theorem}

Inspired by \cite{BKM84}, the first purpose of this paper
is to establish an analog of Beale-Kato-Majda's criterion for
singularities of local smooth solutions to the system
\eqref{eq1.1}. The result shows that the maximum norm
of the vorticity alone controls the breakdown of smooth solutions,
which reveals an important fact that the velocity field $u$ plays a
more dominant role than the charge density functions $v$ and $w$ in the blow-up
theory of  smooth solutions to the system \eqref{eq1.1}.  We follow closely the arguments used in \cite{BKM84}
to accomplish the proof in the following three steps: obtaining $L^{2}$
estimates for the vorticity $\omega$ and $v,w$,
obtaining higher energy estimates for the solution $(u,v,w)$, and
applying the crucial logarithmic Sobolev inequality.

The main results of this paper are as follows:

\begin{theorem}\label{th1.2}
Let $u_{0}\in H^{3}(\mathbb{R}^{3})$ with
$\nabla\cdot u_{0}=0$,  $v_{0}, w_{0}\in L^{1}(\mathbb{R}^{3})\cap
H^{2}(\mathbb{R}^{3})$ and $v_{0}, w_{0}\geq0$.  Let  $T_{*}>0$ be the
maximum existence time such that the system
\eqref{eq1.1} admits an unique smooth solution $(u,v,w)$ satisfying \eqref{eq1.8}
for any $0<T<T_{*}$. If $T_{*}<\infty$, then
\begin{align}\label{eq1.9}
\int_{0}^{T_{*}}\|\omega(\cdot,t)\|_{L^{\infty}}\;dt=\infty.
\end{align}
In particular,
\begin{align*}
  \limsup_{t\nearrow T_{*}}\|\omega(\cdot,t)\|_{L^{\infty}}=\infty.
\end{align*}
\end{theorem}

\noindent\textbf{Remark 1.1}  As a byproduct of the proof of Theorem \ref{th1.1}, we obtain that under the assumptions of Theorem \ref{th1.2}, if for some $0<T<\infty$ such that
\begin{equation*}
\int_{0}^{T}\|\omega(\cdot,t)\|_{L^{\infty}}\;dt<\infty,
\end{equation*}
then the solution $(u,v,w)$ can be extended past the time $t=T$.

Motivated by the blow-up criteria \eqref{eq1.3}--\eqref{eq1.7} for the Navier-Stokes equations \eqref{eq1.2}, the second purpose of this paper is to generalize these blow-up criteria to the system
\eqref{eq1.1}.

\begin{theorem}\label{th1.3}
Let $u_{0}\in H^{3}(\mathbb{R}^{3})$ with
$\nabla\cdot u_{0}=0$,  $v_{0}, w_{0}\in L^{1}(\mathbb{R}^{3})\cap
H^{2}(\mathbb{R}^{3})$ and $v_{0}, w_{0}\geq0$.   Assume that $T_{*}>0$
is the maximal existence time such that the system \eqref{eq1.1} admits
an unique solution $(u,v,w)$ satisfying \eqref{eq1.8}  for any
$0<T<T_{*}$. Then $T_{*}<+\infty$ if and only if one of the following conditions holds:
\begin{align}\label{eq1.10}
   \int_{0}^{T_{*}}\|u(\cdot,t)\|_{L^{p}}^{q}dt=\infty\ \ \ \text{ with
  }\ \ \ \frac{2}{q}+\frac{3}{p}= 1\ \ \ \text{and}\ \ \ 3<p\leq \infty,
\end{align}
\begin{align}\label{eq1.11}
  \int_{0}^{T_{*}}\|\nabla u(\cdot,t)\|_{L^{p}}^{q}dt=\infty\ \  \ \text{ with
  }\ \ \ \frac{2}{q}+\frac{3}{p}=2\ \ \ \text{and}\ \ \ \frac{3}{2}<p\leq \infty.
\end{align}
\end{theorem}

\begin{theorem}\label{th1.4}
Let $u_{0}\in H^{3}(\mathbb{R}^{3})$ with
$\nabla\cdot u_{0}=0$,  $v_{0}, w_{0}\in L^{1}(\mathbb{R}^{3})\cap
H^{2}(\mathbb{R}^{3})$ and $v_{0}, w_{0}\geq0$.  Assume that $T_{*}>0$
is the maximal existence time such that the system \eqref{eq1.1} admits
an unique solution $(u,v,w)$ satisfying \eqref{eq1.8}  for any
$0<T<T_{*}$. Then $T_{*}<+\infty$ if and only if  the following condition holds:
\begin{align}\label{eq1.12}
  \int_{0}^{T_{*}}\|\nabla_{h} u^{h}(\cdot,t)\|_{\dot{B}^{0}_{p,\frac{2p}{3}}}^{q}dt=\infty\ \ \text{with}
  \ \ \ \frac{2}{q}+\frac{3}{p}=2\ \ \ \text{and}\ \ \ \frac{3}{2}<p\leq\infty,
\end{align}
where $\nabla_{h}=(\partial_{1}, \partial_{2})$ and $u^{h}=(u^{1}, u^{2})$.
\end{theorem}

\noindent\textbf{Remark 1.2} It is clear that when $p=\infty$, the condition \eqref{eq1.12} becomes $\int_{0}^{T_{*}}\|\nabla_{h} u^{h}\|_{\dot{B}^{0}_{\infty,\infty}}dt=\infty$, thus Theorem \ref{th1.4} can be regarded as an extension of the blow up criteria \eqref{eq1.5}--\eqref{eq1.7}, and is even new for the Navier-Stokes equations \eqref{eq1.2}.

Before ending this section, we recall the definition of the homogeneous Besov spaces, for more details, see \cite{BCD11} and \cite{T83}. Let $\phi: \mathbb{R}^{3}\rightarrow[0,1]$ be a smooth cut-off function which equals one on the ball
$\mathcal{B}(0, \frac{5}{4})\stackrel{def}{=}\{\xi\in\mathbb{R}^3:\  |\xi|\leq \frac{5}{4}\}$ and equals zero outside the ball $\mathcal{B}(0, \frac{3}{2})$. Write
$$
  \varphi\stackrel{def}{=}\phi(\xi)-\phi(2\xi) \ \ \ \text{and}\
   \ \ \varphi_{j}(\xi)\stackrel{def}{=}\varphi(2^{-j}\xi).
$$
Then we define
for all $j\in \mathbb{Z}$,
\begin{align*}
  \Delta_{j}f\stackrel{def}{=}\varphi(2^{-j}D)f .
\end{align*}
We have the following formal decomposition:
\begin{equation}\label{eq1.13}
  f=\sum_{j\in\mathbb{Z}}\Delta_{j}f \ \ \text{for}\ \
  f\in\mathcal{S}'(\mathbb{R}^{3})/\mathcal{P}(\mathbb{R}^{3}),
\end{equation}
where $\mathcal{P}(\mathbb{R}^{3})$ is the set of polynomials (see \cite{BCD11}).
Let $s\in \mathbb{R}$, $1\leq p,r\leq\infty$ and $f\in\mathcal{S}'(\mathbb{R}^{3})$, the space of temperate distributions, the norms in homogeneous Besov spaces are defined as follows:
\begin{equation*}
  \|f\|_{\dot{B}^{s}_{p,r}}\stackrel{def}{=} \begin{cases} \left(\sum_{j\in\mathbb{Z}}2^{jsr}\|\Delta_{j}f\|_{L^{p}}^{r}\right)^{\frac{1}{r}}
  \ \ &\text{for}\ \ 1\leq r<\infty,\\
  \sup_{j\in\mathbb{Z}}2^{js}\|\Delta_{j}f\|_{L^{p}}\ \
  &\text{for}\ \
  r=\infty.
 \end{cases}
\end{equation*}
Then the homogeneous Besov
space $\dot{B}^{s}_{p,r}(\mathbb{R}^{3})$ is defined by
\begin{itemize}
\item For $s<\frac{3}{p}$ (or $s=\frac{3}{p}$ if $r=1$), we define
\begin{equation*}
  \dot{B}^{s}_{p,r}(\mathbb{R}^{3})\stackrel{def}{=}\Big\{f\in \mathcal{S}'(\mathbb{R}^{3}):\ \
  \|f\|_{\dot{B}^{s}_{p,r}}<\infty\Big\}.
\end{equation*}
\item If $k\in\mathbb{N}$ and $\frac{3}{p}+k\leq s<\frac{3}{p}+k+1$ (or $s=\frac{3}{p}+k+1$ if $r=1$), then $\dot{B}^{s}_{p,r}(\mathbb{R}^{3})$
is defined as the subset of distributions $f\in\mathcal{S}'(\mathbb{R}^{3})$ such that $\partial^{\beta}f\in\mathcal{S}'(\mathbb{R}^{3})$
whenever $|\beta|=k$.
\end{itemize}

Another classical results that will be used here are the so-called Bernstein's inequalities (cf. \cite{BCD11}): there exists a constant $C$ such that
for every $f\in\mathcal{S}'(\mathbb{R}^{3})$, any nonnegative integer $k$  and any
couple of real numbers $(p,q)$ with $1\leq p\leq q\leq\infty$, we have
\begin{align}\label{eq1.14}
  &\operatorname{supp} \widehat{f}\subset\{|\xi|\leq2^{j}\}\ \ \Longrightarrow\ \
  \sup_{|\alpha|=k}\|\partial^{\alpha}f\|_{L^{q}}\leq
  C2^{jk+3j(\frac{1}{p}-\frac{1}{q})}\|f\|_{L^{p}},\\
  \label{eq1.15}
   &\operatorname{supp}\hat{f}\subset\{|\xi|\approx 2^{j}\} \ \ \Longrightarrow\ \   C^{-1-k}2^{jk}\|f\|_{L^{p}}\leq
   \sup_{|\alpha|=k}\|\partial^{\alpha}f\|_{L^{p}}\leq  C^{1+k}2^{jk}\|f\|_{L^{p}},
\end{align}
where $\hat{f}$ represents the Fourier transform of $f$.

The rest of this paper is organized as follows: We shall present the proof of Theorem \ref{th1.2} in the Section 2, while the proof of Theorem \ref{th1.3} will be presented in Section 3 and Theorem \ref{th1.4} in Section 4. Throughout the paper, we denote by $C$ and $C_{i}$ ($i=0,1, \cdots$) the harmless positive constants, which may depend on the initial data and its value may change from line to line, the special dependence will be pointed out explicitly in the text if necessary.

\section{Proof of Theorem \ref{th1.2}}

We prove Theorem \ref{th1.2} by contradiction. It suffices to show that if the condition \eqref{eq1.9}  does not hold, i.e.,
there exists a constant $K>0$ such that
\begin{align}\label{eq2.1}
\int_{0}^{T_{*}}\|\omega(\cdot,t)\|_{L^{\infty}}dt\leq K,
\end{align}
 then one can prove that
\begin{equation}\label{eq2.2}
  \sup_{0\leq t\leq T_{*}}\big(\|u(\cdot,t)\|_{H^{3}}+\|(v(\cdot,t), w(\cdot,t))\|_{H^{2}}\big)\leq C_{0},
\end{equation}
where $C_{0}$ is a constant depending only on $\|u_0\|_{H^{3}}$,
$\|(v_0,w_0)\|_{L^{1}\cap H^{2}}$, $T_{*}$  and $K$.
This implies that the solution $(u,v,w)$ can be extended beyond the time $T_{*}$, which contradicts to the maximality of $T_{*}$.

Let us first recall the
following two important inequalities that play an important roles in the proof of Theorem \ref{th1.2}, for its proof, see \cite{KP88} and
\cite{BKM84}, respectively.

\begin{lemma}\label{le2.1} {\em(\cite{KP88})}
If $s>0$ and $1<p<\infty$, then
\begin{equation}\label{eq2.3}
   \|\Lambda^{s}(fg)-f\Lambda^{s}g\|_{L^{p}}\leq
   C\big(\|\Lambda
   f\|_{L^{\infty}}\|\Lambda^{s-1}g\|_{L^{p}}+\|\Lambda^{s}f\|_{L^{p}}\|g\|_{L^{\infty}}\big),
\end{equation}
where $\Lambda=(-\Delta)^{\frac{1}{2}}$, and $C$ is a constant
depending only on $s$ and $p$.
\end{lemma}

\begin{lemma}\label{le2.2} {\em(\cite{BKM84})}
For all $u\in H^{3}(\mathbb{R}^3)$, $\omega=\nabla\times u$, there exists a
constant $C$ such that
\begin{equation}\label{eq2.4}
   \|\nabla u\|_{L^{\infty}}\leq
  C\big(1+\|\omega\|_{L^{2}}+\|\omega\|_{L^{\infty}}\ln(e+\|u\|_{H^{3}})\big).
\end{equation}
\end{lemma}

Next we establish the $L^{2}$
bounds for the vorticity $\omega$, $v$ and $w$ under the assumptions of Theorem \ref{th1.2} and the condition \eqref{eq2.1}.

\begin{lemma}\label{le2.3}
Assume that $u_{0}\in H^{1}(\mathbb{R}^3)$ with $\nabla\cdot u_{0}=0$, $(v_{0},w_{0})\in L^{1}(\mathbb{R}^3)\cap L^{2}(\mathbb{R}^3)$  and $v_{0}, w_{0}\geq0$. Let $(u,v,w)$ be the
corresponding local smooth solution to the system
\eqref{eq1.1} on $[0,T)$ for some $0<T<\infty$. If
\begin{align}\label{eq2.5}
\int_{0}^{T}\|\omega(\cdot,t)\|_{L^{\infty}}\;dt=K<\infty,
\end{align}
then we have
\begin{align}\label{eq2.6}
  \sup_{0\leq t\leq T}\Big(\|v(\cdot,t)\|_{L^{2}}^{2}+\|w(\cdot,t)\|_{L^{2}}^{2}\Big)+2\int_{0}^{T}\Big(\|\nabla
  v(\cdot,t)\|_{L^{2}}^{2}+\|\nabla
  w(\cdot,t)\|_{L^{2}}^{2}\Big)dt\leq C_{0}
\end{align}
and
\begin{align}\label{eq2.7}
  \sup_{0\leq t\leq T}\|\omega(\cdot,t)\|_{L^{2}}^{2}+\int_{0}^{T}\|\nabla
  \omega(\cdot,t)\|_{L^{2}}^{2}dt\leq C_{0},
\end{align}
where $C_{0}$ is a constant depending only on $\|u_{0}\|_{H^{1}}$, $\|(v_{0},w_{0})\|_{L^{1}\cap L^{2}}$, $T$ and $K$.
\end{lemma}
\begin{proof} We first recall from \cite{S09} that if $v_0$ and $w_0$  are non-negative, then $v$ and $w$ are also non-negative.
Next, in order to derive
$L^{2}$ bounds of $v$ and $w$, we introduce two symmetries: $\zeta\stackrel{def}{=}v+w$ and $\eta\stackrel{def}{=}v-w$. Then we infer from
\eqref{eq1.1} that  $\zeta$
and $\eta$ satisfy the following system:
\begin{equation}\label{eq2.8}
\begin{cases}
  \partial_{t}\zeta+(u\cdot\nabla) \zeta=\Delta \zeta-\nabla\cdot(\eta\nabla\Psi),\\
  \partial_{t}\eta+(u\cdot\nabla) \eta=\Delta\eta-\nabla\cdot (\zeta\nabla\Psi),\\
  \Delta\Psi=\eta,\\
  \zeta_{0}=v_{0}+w_{0},\ \ \eta_{0}=v_{0}-w_{0}.
\end{cases}
\end{equation}
Multiplying the first equation of \eqref{eq2.8} by $\zeta$ and integrating over $\mathbb{R}^{3}$, one has
\begin{equation*}
  \frac{1}{2}\frac{d}{dt}\|\zeta\|_{L^{2}}^{2}+\|\nabla
  \zeta\|_{L^{2}}^{2}+\int_{\mathbb{R}^{3}}\zeta\nabla \eta\cdot\nabla\Psi dx+\int_{\mathbb{R}^{3}}\zeta\eta\Delta\Psi
  dx=0,
\end{equation*}
where we have used the fact
$$
\int_{\mathbb{R}^{3}}(u\cdot\nabla) \zeta
\zeta dx=\frac{1}{2}\int_{\mathbb{R}^{3}}u\cdot\nabla (\zeta^2)dx=0
$$
due to  $\nabla\cdot u=0$.
The third equation of \eqref{eq2.8} yields that
\begin{equation}\label{eq2.9}
  \frac{1}{2}\frac{d}{dt}\|\zeta\|_{L^{2}}^{2}+\|\nabla
  \zeta\|_{L^{2}}^{2}+\int_{\mathbb{R}^{3}}\zeta\nabla \eta\cdot\nabla\Psi dx+\int_{\mathbb{R}^{3}}\zeta\eta^{2}
  dx=0.
\end{equation}
It can be done analogously for $\eta$, we get
\begin{equation}\label{eq2.10}
  \frac{1}{2}\frac{d}{dt}\|\eta\|_{L^{2}}^{2}+\|\nabla
  \eta\|_{L^{2}}^{2}+\int_{\mathbb{R}^{3}}\eta\nabla \zeta\cdot\nabla\Psi dx+\int_{\mathbb{R}^{3}}\zeta\eta^{2}
  dx=0.
\end{equation}
Adding \eqref{eq2.9} and \eqref{eq2.10} together, after integration
by parts, we obtain
\begin{equation}\label{eq2.11}
  \frac{1}{2}\frac{d}{dt}\Big(\|\zeta\|_{L^{2}}^{2}+\|\eta\|_{L^{2}}^{2}\Big)+\|\nabla
  \zeta\|_{L^{2}}^{2}+\|\nabla
  \eta\|_{L^{2}}^{2}+\int_{\mathbb{R}^{3}}\zeta\eta^{2}
  dx=0.
\end{equation}
Integrating \eqref{eq2.11} over $[0,t]$ for all $0<t\leq T$ implies that
\begin{align*}
  \|\zeta(t)\|_{L^{2}}^{2}+\|\eta(t)\|_{L^{2}}^{2}&+2\int_{0}^{t}\Big(\|\nabla
  \zeta(\tau)\|_{L^{2}}^{2}+\|\nabla
  \eta(\tau)\|_{L^{2}}^{2}\Big)d\tau\\
  &+2\int_{0}^{t}\int_{\mathbb{R}^{3}}\zeta(x,\tau)\eta^{2}(x,\tau)
  dxd\tau=\|\zeta_{0}\|_{L^{2}}^{2}+\|\eta_{0}\|_{L^{2}}^{2},
\end{align*}
which gives us to
\begin{align}\label{eq2.12}
  \|v(t)\|_{L^{2}}^{2}&+\|w(t)\|_{L^{2}}^{2}+2\int_{0}^{t}\Big(\|\nabla
  v(\tau)\|_{L^{2}}^{2}+\|\nabla
  w(\tau)\|_{L^{2}}^{2}\Big)d\tau\nonumber\\&+2\int_{0}^{t}\int_{\mathbb{R}^{3}}(v(x,\tau)+w(x,\tau))(v(x,\tau)-w(x,\tau))^{2}
  dxd\tau=\|v_{0}\|_{L^{2}}^{2}+\|w_{0}\|_{L^{2}}^{2}.
\end{align}
Since $v$ and $w$ are non-negative, we infer from \eqref{eq2.12} that  for all $0<t\leq T$,
\begin{align*}
  \|v(t)\|_{L^{2}}^{2}+\|w(t)\|_{L^{2}}^{2}+2\int_{0}^{t}\Big(\|\nabla
  v(\tau)\|_{L^{2}}^{2}+\|\nabla
  w(\tau)\|_{L^{2}}^{2}\Big)d\tau\leq \|v_{0}\|_{L^{2}}^{2}+\|w_{0}\|_{L^{2}}^{2}.
\end{align*}
This proves  \eqref{eq2.6}.

To derive $L^{2}$ bound of $u$, multiplying
 the first equation of \eqref{eq1.1} by $u$, after integration by
parts, by the fifth equation of \eqref{eq1.1}, it can be easily seen
that
\begin{align}\label{eq2.13}
  \frac{1}{2}\frac{d}{dt}\|u\|_{L^{2}}^{2}+\|\nabla
  u\|_{L^{2}}^{2}=\int_{\mathbb{R}^{3}}(v-w)u\cdot\nabla\Psi dx.
\end{align}
On the other hand, we multiply, respectively,  the third and fourth equations of  \eqref{eq1.1} by
$\Psi$ to yield that
\begin{align}\label{eq2.14}
  &\int_{\mathbb{R}^{3}}\Big(\partial_{t}v\Psi+\nabla\cdot(v\nabla\Psi)\Psi-\Delta
  v\Psi+(u\cdot\nabla)
  v\Psi\Big)dx=0,\\
\label{eq2.15}
  &\int_{\mathbb{R}^{3}}\Big(\partial_{t}w\Psi-\nabla\cdot(w\nabla\Psi)\Psi-\Delta
  w\Psi+(u\cdot\nabla)
  w\Psi\Big)dx=0.
\end{align}
Subtracting \eqref{eq2.15} from \eqref{eq2.14} leads to
\begin{equation*}
  \int_{\mathbb{R}^{3}}\Big(\partial_{t}(v-w)\Psi+\nabla\cdot((v+w)\nabla\Psi)\Psi-\Delta (v-w)\Psi
  +(u\cdot\nabla)
  (v-w)\Psi\Big)dx=0,
\end{equation*}
which, after integration by part and using the fifth equation of \eqref{eq1.1} again, gives us to
\begin{equation}\label{eq2.16}
  \frac{1}{2}\frac{d}{dt}\|\nabla\Psi\|_{L^{2}}^{2}+\int_{\mathbb{R}^{3}}(v+w)|\nabla\Psi|^{2}dx+\int_{\mathbb{R}^{3}}|\Delta\Psi|^{2}dx
  +\int_{\mathbb{R}^{3}}(v-w)u\cdot\nabla\Psi dx
  =0.
\end{equation}
Adding \eqref{eq2.13} and \eqref{eq2.16} together, we find that
\begin{equation*}
   \frac{1}{2}\frac{d}{dt}\Big(\|u\|_{L^{2}}^{2}+\|\nabla\Psi\|_{L^{2}}^{2}\Big)+\|\nabla
  u\|_{L^{2}}^{2}+\|\Delta\Psi\|_{L^{2}}^{2}
  +\int_{\mathbb{R}^{3}}(v+w)|\nabla\Psi|^{2}dx=0.
\end{equation*}
Observing that $v$ and $w$ are non-negative, thus we have
\begin{equation}\label{eq2.17}
   \frac{1}{2}\frac{d}{dt}\Big(\|u\|_{L^{2}}^{2}+\|\nabla\Psi\|_{L^{2}}^{2}\Big)+\|\nabla
  u\|_{L^{2}}^{2}+\|\Delta\Psi\|_{L^{2}}^{2}
  \leq0.
\end{equation}
Integrating \eqref{eq2.17} in the time interval $[0,t]$ for any $0<t\leq T$, one obtains that
\begin{equation}\label{eq2.18}
  \|u(t)\|_{L^{2}}^{2}+\|\nabla\Psi(t)\|_{L^{2}}^{2}+2\int_{0}^{t}\Big(\|\nabla
  u(\tau)\|_{L^{2}}^{2}+\|\Delta\Psi(\tau)\|_{L^{2}}^{2}\Big)d\tau\leq
  \|u_{0}\|_{L^{2}}^2+\|\nabla\Psi_{0}\|_{L^{2}}^{2},
\end{equation}
where $\Psi_{0}$ is determined by
$v_{0}$ and $w_{0}$ through the equation $\Delta\Psi_{0}=v_{0}-w_{0}$. By using the Sobolev
embedding relation $\dot{H}^{1}(\mathbb{R}^{3})\hookrightarrow
L^{6}(\mathbb{R}^{3})$, the Young inequality and the
interpolation inequality, we obtain that
\begin{align*}
  \|\nabla\Psi_{0}\|_{L^{2}}^{2}&\leq
  \|\Psi_{0}\|_{L^{6}}\|\Delta\Psi_{0}\|_{L^{\frac{6}{5}}}\leq \|\Psi_{0}\|_{L^{6}}\|(v_{0},
  w_{0})\|_{L^{\frac{6}{5}}}\\
  &\leq C\|\nabla\Psi_{0}\|_{L^{2}}\|(v_{0}, w_{0})\|_{L^{\frac{6}{5}}}
  \leq \frac{1}{2}\|\nabla\Psi_{0}\|_{L^{2}}^{2}+C\|(v_{0}, w_{0})\|_{L^{\frac{6}{5}}}^{2}\\
  &\leq \frac{1}{2}\|\nabla\Psi_{0}\|_{L^{2}}^{2}+C\|(v_{0}, w_{0})\|_{L^{1}}^{\frac{4}{3}}\|(v_{0},
  w_{0})\|_{L^{2}}^{\frac{2}{3}}.
\end{align*}
Hence,  for all $0<t\leq T$, we obtain the
following $L^{2}$ bound of $u$ and $\nabla\Psi$:
\begin{equation}\label{eq2.19}
  \|u(t)\|_{L^{2}}^{2}+\|\nabla\Psi(t)\|_{L^{2}}^{2}+2\int_{0}^{t}\big(\|\nabla
  u(\tau)\|_{L^{2}}^{2}+\|\Delta\Psi(\tau)\|_{L^{2}}^{2}\big)d\tau\leq
  \|u_{0}\|_{L^{2}}^{2}+C_{1},
\end{equation}
where $C_{1}$ is a constant depending only on $\|(v_{0},
w_{0})\|_{L^{1}\cap L^{2}}$.

Finally, we establish the desired estimate for $\omega$. Taking $\nabla\times$  on the first equation
of \eqref{eq1.1}, we see that
\begin{equation}\label{eq2.20}
  \partial_{t}\omega+(u\cdot\nabla)\omega-\Delta \omega=(\omega\cdot\nabla)u+\nabla\times(\Delta \Psi\nabla \Psi).
\end{equation}
Multiplying \eqref{eq2.20} by $\omega$, and integrating over $\mathbb{R}^{3}$, one has
\begin{align}\label{eq2.21}
  \frac{1}{2}\frac{d}{dt}\|\omega\|_{L^{2}}^{2}+\|\nabla \omega\|_{L^{2}}^{2}&=\int_{\mathbb{R}^{3}}(\omega\cdot\nabla)u\cdot \omega dx-\int_{\mathbb{R}^{3}}(\Delta\Psi\nabla \Psi)\cdot(\nabla\times \omega)dx\nonumber\\
  &\overset{\operatorname{def}}{=}I_{1}+I_{2}.
\end{align}
By the Biot-Savart law, we have the
following two equalities:
\begin{align*}
  \frac{\partial u}{\partial x_{j}}=R_{j}(R\times\omega) \ \
  \text{for}\ \ j=1,2,3,
\end{align*}
and
\begin{align*}
  \nabla\times \omega=\nabla\times(\nabla\times u)=\nabla(\nabla\cdot u)-\Delta u
  =-\Delta u,
\end{align*}
where $R=(R_{1}, R_{2}, R_{3})$, $R_{j}=\frac{\partial}{\partial
x_{j}}(-\Delta)^{-\frac{1}{2}}$ ($j=1,2,3$) are the Riesz transforms.
Since the Riesz operators are bounded in $L^{2}(\mathbb{R}^3)$, we have
$$
  \|\nabla u\|_{L^{2}}\leq C\|\omega\|_{L^{2}}\ \ \text{and}\ \  \|\Delta u\|_{L^{2}}\leq C\|\nabla\omega\|_{L^{2}}.
$$
This implies that
\begin{equation}\label{eq2.22}
  I_{1}\leq\Big{|}\int_{\mathbb{R}^{3}}(\omega\cdot\nabla) u\cdot\omega dx\Big{|}\leq
  \|\omega\|_{L^{\infty}}\|\nabla u\|_{L^{2}}\|\omega\|_{L^{2}}\leq
  C\|\omega\|_{L^{\infty}}\|\omega\|_{L^{2}}^{2}.
\end{equation}
Moreover, by using the fifth equation of \eqref{eq1.1},  \eqref{eq2.6} and \eqref{eq2.19},  we bound $I_{2}$ as
\begin{align}\label{eq2.23}
  I_{2}  &\leq C \|\nabla \Psi\|_{L^{4}}\|\Delta \Psi\|_{L^{4}}\|\Delta u\|_{L^{2}}\nonumber\\
  &\leq \frac{1}{2}\|\Delta u\|_{L^{2}}^{2}+C\|(v,w)\|_{L^{4}}^{2}\|\nabla \Psi\|_{L^{4}}^{2}\nonumber\\
  &\leq \frac{1}{2}\|\nabla \omega\|_{L^{2}}^{2}+C\|(v,w)\|_{L^{2}}^{\frac{1}{2}}\|(\nabla v,\nabla w)\|_{L^{2}}^{\frac{3}{2}}\|\nabla\Psi\|_{L^{2}}^{\frac{1}{2}}\|\Delta\Psi\|_{L^{2}}^{\frac{3}{2}}\nonumber\\
  &\leq \frac{1}{2}\|\nabla \omega\|_{L^{2}}^{2}+C\|(v,w)\|_{L^{2}}^{2}\|(\nabla v,\nabla w)\|_{L^{2}}^{\frac{3}{2}}\|\nabla\Psi\|_{L^{2}}^{\frac{1}{2}}\nonumber\\
  &\leq \frac{1}{2}\|\nabla \omega\|_{L^{2}}^{2}+C\|(v,w)\|_{L^{2}}^{2}\|(\nabla v,\nabla w)\|_{L^{2}}^{2}+C\|(v,w)\|_{L^{2}}^{2}\|\nabla\Psi\|_{L^{2}}^{2}\nonumber\\
   &\leq \frac{1}{2}\|\nabla \omega\|_{L^{2}}^{2}+C(\|(\nabla v,\nabla w)\|_{L^{2}}^{2}+1),
\end{align}
where we have used the
following interpolation inequality:
\begin{equation*}
  \|f\|_{L^{4}}\leq \|f\|_{L^{2}}^{\frac{1}{4}}\|
  \nabla f\|_{L^{2}}^{\frac{3}{4}}.
\end{equation*}
Plugging \eqref{eq2.22} and \eqref{eq2.23} into \eqref{eq2.21}, we
obtain
\begin{equation}\label{eq2.24}
  \frac{d}{dt}\|\omega\|_{L^{2}}^{2}+\|\nabla\omega\|_{L^{2}}^{2}\leq
  C\Big(\|\omega\|_{L^{\infty}}+\|(\nabla v,\nabla w)\|_{L^{2}}^{2}+1\Big)\Big(\|\omega\|_{L^{2}}^{2}+1\Big).
\end{equation}
Applying the Gronwall's inequality, we conclude that for all $0<t\leq T$,
\begin{align}\label{eq2.25}
  \|\omega(t)\|_{L^{2}}^{2}+\int_{0}^{t}\|\nabla\omega(\tau)\|_{L^{2}}^{2}d\tau\leq
  C_{2}\exp\Big\{C\int_{0}^{t}\big(\|\omega(\tau)\|_{L^{\infty}}+\|(\nabla v(\tau),\nabla w(\tau))\|_{L^{2}}^{2}+1\big)d\tau\Big\},
\end{align}
where
$C_{2}=\|\omega_{0}\|_{L^{2}}^{2}+1$.
By \eqref{eq2.5}, \eqref{eq2.6} and \eqref{eq2.25}, we get
\eqref{eq2.7}. The proof of Lemma \ref{le2.3} is complete.
\end{proof}

\medskip

Now, let us derive higher order derivative bounds of the solution
$(u,v,w)$. Taking $\nabla\Delta$ on the first equation of \eqref{eq1.1}, multiplying the
resulting with $\nabla\Delta u$ and integrating over $\mathbb{R}^{3}$, observing
that the pressure $\Pi$ can be eliminated by the condition
$\nabla\cdot u=0$, one obtains
\begin{align}\label{eq2.26}
  \frac{1}{2}\frac{d}{dt}\|\nabla\Delta
  u\|_{L^{2}}^{2}+\|\Delta^{2}u\|_{L^{2}}^{2}&=-\int_{\mathbb{R}^{3}}\nabla\Delta((u\cdot\nabla)
  u)\cdot\nabla\Delta u dx+\int_{\mathbb{R}^{3}}\nabla\Delta(\Delta\Psi\nabla\Psi)\cdot\nabla\Delta u dx\nonumber\\
  &\stackrel{def}{=}II_{1}+II_{2}.
\end{align}
Since $\nabla\cdot u=0$, it follows from Lemma \ref{le2.1} that
\begin{align}\label{eq2.27}
  II_{1}&=-\int_{\mathbb{R}^{3}}\big[\nabla\Delta((u\cdot\nabla)
  u)-(u\cdot\nabla)\nabla\Delta u\big]\cdot\nabla\Delta u dx\nonumber\\
  &\leq C\|\nabla\Delta((u\cdot\nabla)
  u)-(u\cdot\nabla)\nabla\Delta u\|_{L^{2}}\|\nabla\Delta
  u\|_{L^{2}}\nonumber\\
  &\leq C\|\nabla u\|_{L^{\infty}}\|\nabla\Delta
  u\|_{L^{2}}^{2}.
\end{align}
For $II_{2}$, by using the fifth equation of \eqref{eq1.1}, \eqref{eq2.6}, \eqref{eq2.7} and \eqref{eq2.19}, one has
\begin{align}\label{eq2.28}
  II_{2}&=-\int_{\mathbb{R}^{3}}\Delta((v-w)\nabla\Psi)\cdot\Delta^{2}u dx\nonumber\\
  &\leq \frac{1}{2}\|\Delta^{2}u\|_{L^{2}}^{2}+C\|\Delta((v-w)\nabla\Psi)\|_{L^{2}}^{2}\nonumber\\
  &\leq \frac{1}{2}\|\Delta^{2}u\|_{L^{2}}^{2}+C\Big(\|(\Delta v-\Delta w)\nabla\Psi\|_{L^{2}}^{2}+2\|(\nabla v-\nabla w)\nabla^{2}\Psi\|_{L^{2}}^{2}+\|(v-w)\nabla\Delta\Psi\|_{L^{2}}^{2}\Big)\nonumber\\
  &\leq \frac{1}{2}\|\Delta^{2}u\|_{L^{2}}^{2}+C\Big(\|(\Delta v, \Delta w)\|_{L^{3}}^{2}\|\nabla\Psi\|_{L^{6}}^{2}+\|(v,w)\|_{L^{3}}^{2}\|(\nabla v,\nabla w)\|_{L^{6}}^{2}\Big)\nonumber\\
  &\leq \frac{1}{2}\|\Delta^{2}u\|_{L^{2}}^{2}+C\Big(\|(\Delta v, \Delta w)\|_{L^{2}}\|(\nabla\Delta v, \nabla\Delta w)\|_{L^{2}}+\|(\nabla v,\nabla w)\|_{L^{2}}\|(\Delta v, \Delta w)\|_{L^{2}}^{2}\Big)\nonumber\\
  &\leq \frac{1}{2}\|\Delta^{2}u\|_{L^{2}}^{2}+\frac{1}{6}\|(\nabla\Delta v, \nabla\Delta w)\|_{L^{2}}^{2}+C\Big(1+\|(\nabla v, \nabla w)\|_{L^{2}}^{2}\Big)\|(\Delta v, \Delta w)\|_{L^{2}}^{2},
\end{align}
where we have used the interpolation inequality:
\begin{equation*}
  \|f\|_{L^{3}}\leq \|f\|_{L^{2}}^{\frac{1}{2}}\|
  \nabla f\|_{L^{2}}^{\frac{1}{2}},
\end{equation*}
and the following Hardy-Littlewood-Sobolev inequality with $p=2$:
\begin{equation*}
  \|\nabla(-\Delta)^{-1}f\|_{L^{\frac{3p}{3-p}}}\leq
  C\|f\|_{L^{p}} \ \ \text{for any}\ \  1<p<3.
\end{equation*}
Putting \eqref{eq2.27} and \eqref{eq2.28} together, we deduce that
\begin{align}\label{eq2.29}
  \frac{d}{dt}\|\nabla\Delta
  u\|_{L^{2}}^{2}&+\|\Delta^{2}u\|_{L^{2}}^{2}\leq \frac{1}{3}\|(\nabla\Delta v, \nabla\Delta w)\|_{L^{2}}^{2}\nonumber\\
  &+C\Big(1+\|\nabla u\|_{L^{\infty}}+\|(\nabla v, \nabla w)\|_{L^{2}}^{2}\Big)\Big(\|\nabla\Delta
  u\|_{L^{2}}^{2}+\|(\Delta v, \Delta w)\|_{L^{2}}^{2}\Big).
\end{align}
Taking $\Delta$ to the third equation of \eqref{eq1.1}, then multiplying $\Delta v$, after integration by parts, we see that
\begin{align}\label{eq2.30}
  \frac{1}{2}\frac{d}{dt}\|\Delta
  v\|_{L^{2}}^{2}+\|\nabla\Delta v\|_{L^{2}}^{2}&=-\int_{\mathbb{R}^{3}}\Delta((u\cdot\nabla)
  v)\Delta v dx-\int_{\mathbb{R}^{3}}\Delta\nabla\cdot(v\nabla\Psi)\Delta v dx\nonumber\\
  &\stackrel{def}{=}III_{1}+III_{2}.
\end{align}
Applying Lemma \ref{le2.3}, the terms $III_{1}$ and $III_{2}$ can be estimated as follows:
\begin{align*}
  III_{1}&=\int_{\mathbb{R}^3}\nabla((u\cdot\nabla) v)\cdot\nabla\Delta vdx\nonumber\\
  &\leq \frac{1}{8}\|\nabla\Delta v\|_{L^{2}}^{2}+C\big(\|(\nabla u\cdot\nabla)v\|_{L^{2}}^{2}+\|(u\cdot\nabla)\nabla v\|_{L^{2}}^{2}\big)\nonumber\\
  &\leq \frac{1}{8}\|\nabla\Delta v\|_{L^{2}}^{2}+C\big(\|\nabla u\|_{L^{\infty}}^{2}\|\nabla v\|_{L^{2}}^{2}+\|u\|_{L^{6}}^{2}\|\Delta v\|_{L^{3}}^{2}\big)\nonumber\\
  &\leq \frac{1}{8}\|\nabla\Delta v\|_{L^{2}}^{2}+C\big(\|\nabla v\|_{L^{2}}^{2}(1+\|\nabla \Delta u\|_{L^{2}}^{2})+\|\nabla u\|_{L^{2}}^{2}\|\Delta v\|_{L^{2}}\|\nabla\Delta v\|_{L^{2}}\big)\nonumber\\
  &\leq \frac{1}{6}\|\nabla\Delta v\|_{L^{2}}^{2}+C\big(1+\|\nabla v\|_{L^{2}}^{2}\big)\big(1+\|\nabla \Delta u\|_{L^{2}}^{2}+\|\Delta v\|_{L^{2}}^{2}\big)
\end{align*}
and
\begin{align*}
  III_{2}&=\int_{\mathbb{R}^3}\Delta(v\nabla \Psi)\cdot\nabla\Delta vdx\nonumber\\
  &\leq \frac{1}{8}\|\nabla\Delta v\|_{L^{2}}^{2}+C\big(\|\Delta v\nabla\Psi\|_{L^{2}}^{2}+\|\nabla v\nabla^{2}\Psi\|_{L^{2}}^{2}+\|v\nabla \Delta\Psi)\|_{L^{2}}^{2}\big)\nonumber\\
  &\leq \frac{1}{8}\|\nabla\Delta v\|_{L^{2}}^{2}+C\big(\|\Delta v\|_{L^{3}}^{2}\|\nabla \Psi\|_{L^{6}}^{2}+\|\nabla v\|_{L^{6}}^{2}\|\nabla^{2} \Psi\|_{L^{3}}^{2}+\|v\|_{L^{3}}^{2}\|\nabla(v-w)\|_{L^{6}}^{2}\big)\nonumber\\
  &\leq \frac{1}{6}\|\nabla\Delta v\|_{L^{2}}^{2}+C\big(1+\|(\nabla v, \nabla w)\|_{L^{2}}^{2}\big)\big(1+\|(\Delta v, \Delta w)\|_{L^{2}}^{2}\big),
\end{align*}
which give us to
\begin{align}\label{eq2.31}
  \frac{d}{dt}\|\Delta
  v\|_{L^{2}}^{2}+\frac{4}{3}\|\nabla\Delta v\|_{L^{2}}^{2}\leq C\big(1+\|(\nabla v, \nabla w)\|_{L^{2}}^{2}\big)\big(1+\|\nabla \Delta u\|_{L^{2}}^{2}+\|(\Delta v, \Delta w)\|_{L^{2}}^{2}\big).
\end{align}
It can be done analogous for $w$, thus we get
\begin{align}\label{eq2.32}
  \frac{d}{dt}\|\Delta
  w\|_{L^{2}}^{2}+\frac{4}{3}\|\nabla\Delta w\|_{L^{2}}^{2}\leq C\big(1+\|(\nabla v, \nabla w)\|_{L^{2}}^{2}\big)\big(1+\|\nabla \Delta u\|_{L^{2}}^{2}+\|(\Delta v, \Delta w)\|_{L^{2}}^{2}\big).
\end{align}

In the final step, we complete the proof of Theorem \ref{th1.2} by applying the logarithmical Sobolev inequality \eqref{eq2.2} and the
Gronwall's inequality. Set
$$
  Y(t)=e+\|u(t)\|_{H^{3}}^{2}+\|(v(t),w(t))\|_{H^{2}}^{2}.
$$
Then by  \eqref{eq2.6},  \eqref{eq2.19}, \eqref{eq2.29}, \eqref{eq2.31} and \eqref{eq2.32}, we see that
\begin{equation}\label{eq2.33}
  \frac{dY(t)}{dt}\leq C\Big(1+\|\nabla u\|_{L^{\infty}}+\|(\nabla v, \nabla w)\|_{L^{2}}^{2}\Big)Y(t).
\end{equation}
On the other hand,  using \eqref{eq2.7} and the logarithmic Sobolev inequality in
Lemma \ref{le2.2}, we get
\begin{equation}\label{eq2.34}
  \|\nabla u\|_{L^{\infty}}\leq
  C\big(1+\|\omega\|_{L^{\infty}}\ln(e+\|u\|_{H^{3}})\big).
\end{equation}
Combining \eqref{eq2.33}, \eqref{eq2.34} with the fact $\ln m(t)\geq
1$ yield that
\begin{align}\label{eq2.35}
  \frac{d}{dt}\ln Y(t)\leq C\Big(1+\|\omega\|_{L^{\infty}}+\|(\nabla v, \nabla w)\|_{L^{2}}^{2}\Big)\ln Y(t).
\end{align}
By applying the Gronwall's inequality, we obtain
\begin{align}\label{eq2.36}
  Y(t)\leq \exp\Big\{\ln Y(0)\exp\big(C(1+t)+C\int_{0}^{t}\|\omega(\tau)\|_{L^{\infty}}d\tau\big)\Big\}
\end{align}
for any $0<t\leq T_{*}$, where $Y(0)=e+\|u_{0}\|_{H^{3}}^{2}+\|(v_0,w_0)\|_{H^{2}}^{2}$.
Hence, by \eqref{eq2.36}, we obtain
$$
  \sup_{0\leq t\leq T_{*}}\big(\|u(\cdot,t)\|_{H^{3}}+\|(v(\cdot,t),w(\cdot,t))\|_{H^{2}}\big)\leq
  C_0,
$$
where $C_{0}$ is a constant depending only on $\|u_0\|_{H^{3}}$,
$\|(v_0,w_0)\|_{L^{1}\cap H^{2}}$, $T_{*}$  and $K$.
The proof of Theorem \ref{th1.2} is complete.

\section{Proof of Theorem \ref{th1.3}}

In this section we prove Theorem \ref{th1.3}. Based on the Beale-Kato-Majda's criterion \eqref{eq1.9} in Theorem \ref{th1.2}, we need only to show that if one of the conditions  \eqref{eq1.10} and \eqref{eq1.11} does not hold, then there exists a constant $C$ such that
\begin{equation}\label{eq3.1}
  \int_{0}^{T_{*}}\|u(\cdot,t)\|_{H^{3}}dt\leq C.
\end{equation}
 By the Sobolev embedding  $H^{2}(\mathbb{R}^{3})\rightarrow L^{\infty}(\mathbb{R}^{3})$, \eqref{eq3.1} gives us to the fact
$\int_{0}^{T_{*}}\|\omega(\cdot,t)\|_{L^{\infty}}dt\leq C$, which completes the proof of Theorem \ref{th1.3} by the Beale-Kato-Majda's criterion \eqref{eq1.9} in Theorem \ref{th1.2}.

\begin{lemma}\label{le3.1}
Assume that $u_{0}\in H^{1}(\mathbb{R}^3)$ with $\nabla\cdot u_{0}=0$, $(v_{0},w_{0})\in L^{1}(\mathbb{R}^3)\cap L^{2}(\mathbb{R}^3)$  and $v_{0}, w_{0}\geq0$. Let $(u,v,w)$ be the
corresponding local smooth solution to the system
\eqref{eq1.1} on $[0,T)$ for some $0<T<\infty$. If
there exists a constant $K>0$ such that one of the following conditions holds:
\begin{align}\label{eq3.2}
   \int_{0}^{T}\|u(\cdot,t)\|_{L^{p}}^{q}dt\leq K \ \ \ \text{ with
  }\ \ \ \frac{2}{q}+\frac{3}{p}= 1\ \ \ \text{and}\ \ \ 3<p\leq \infty,
\end{align}
\begin{align}\label{eq3.3}
   \int_{0}^{T}\|\nabla u(\cdot,t)\|_{L^{p}}^{q}dt\leq K \ \ \ \text{ with
  }\ \ \ \frac{2}{q}+\frac{3}{p}= 2\ \ \ \text{and}\ \ \ \frac{3}{2}<p\leq \infty,
\end{align}
then we have
\begin{align}\label{eq3.4}
  \sup_{0\leq t\leq T}\|\nabla u(\cdot,t)\|_{L^{2}}^{2}+\int_{0}^{T}\|\Delta u(\cdot,t)\|_{L^{2}}^{2}dt\leq C_{3},
\end{align}
where $C_{3}$ is a constant depending only on $\|u_{0}\|_{H^{1}}$, $\|(v_{0},w_{0})\|_{L^{1}\cap L^{2}}$, $T$  and $K$.
\end{lemma}
\begin{proof}
 Multiplying the first equation of \eqref{eq1.1} by $\Delta u$,  after integration by parts, we deduce that
\begin{align}\label{eq3.5}
  \frac{1}{2}\frac{d}{dt}\|\nabla u\|_{L^{2}}^{2}+\|\Delta u\|_{L^{2}}^{2}=-\int_{\mathbb{R}^{3}}u\cdot\nabla u\cdot\Delta udx+\int_{\mathbb{R}^{3}}\Delta\Psi\nabla\Psi\cdot\Delta udx\stackrel{def}{=}J_{1}+J_{2}.
\end{align}
Under the assumption  \eqref{eq3.2}, using the H\"{o}lder inequality, $J_{1}$ can be estimated as
\begin{align}\label{eq3.6}
  J_{1}&\leq C \|u\|_{L^{p}}\|\nabla u\|_{L^{\frac{2p}{p-2}}}\|\Delta u\|_{L^{2}}
  \leq C \|u\|_{L^{p}}\|\nabla u\|_{L^{2}}^{1-\frac{3}{p}}
  \|\Delta u\|_{L^{2}}^{1+\frac{3}{p}}\nonumber\\
  &\leq \frac{1}{4}\|\Delta u\|_{L^{2}}^{2}+C \|u\|_{L^{p}}^{q}\|\nabla u\|_{L^{2}}^{2},
\end{align}
where $q=\frac{2p}{p-3}$, while under the assumption \eqref{eq3.3}, $J_{1}$ can be done as
\begin{align}\label{eq3.7}
  J_{1}&=\sum_{i,k=1}^{3}\int_{\mathbb{R}^{3}}\partial_{k}u^{i}\partial_{i} u\partial_{k} udx\leq C \|\nabla u\|_{L^{p}}\|\nabla u\|_{L^{\frac{2p}{p-1}}}^{2}\nonumber\\
  &\leq C \|\nabla u\|_{L^{p}}\|\nabla u\|_{L^{2}}^{2-\frac{3}{p}}
  \|\Delta u\|_{L^{2}}^{\frac{3}{p}}
  \leq \frac{1}{4}\|\Delta u\|_{L^{2}}^{2}+C \|\nabla u\|_{L^{p}}^{\widetilde{q}}\|\nabla u\|_{L^{2}}^{2},
\end{align}
where $\widetilde{q}=\frac{2p}{2p-3}$. As for $J_{2}$, by using the fifth equation of \eqref{eq1.1},  \eqref{eq2.6} and \eqref{eq2.19},  we get
\begin{align}\label{eq3.8}
 J_{2}  &\leq C \|\nabla \Psi\|_{L^{4}}\|\Delta \Psi\|_{L^{4}}\|\Delta u\|_{L^{2}}\nonumber\\
  &\leq \frac{1}{4}\|\Delta u\|_{L^{2}}^{2}+C\|(v,w)\|_{L^{4}}^{2}\|\nabla \Psi\|_{L^{4}}^{2}\nonumber\\
  &\leq \frac{1}{4}\|\Delta u\|_{L^{2}}^{2}+C\|(v,w)\|_{L^{2}}^{\frac{1}{2}}\|(\nabla v,\nabla w)\|_{L^{2}}^{\frac{3}{2}}\|\nabla\Psi\|_{L^{2}}^{\frac{1}{2}}\|\Delta\Psi\|_{L^{2}}^{\frac{3}{2}}\nonumber\\
  &\leq \frac{1}{4}\|\Delta u\|_{L^{2}}^{2}+C\|(v,w)\|_{L^{2}}^{2}\|(\nabla v,\nabla w)\|_{L^{2}}^{\frac{3}{2}}\|\nabla\Psi\|_{L^{2}}^{\frac{1}{2}}\nonumber\\
  &\leq \frac{1}{4}\|\Delta u\|_{L^{2}}^{2}+C\|(v,w)\|_{L^{2}}^{2}\|(\nabla v,\nabla w)\|_{L^{2}}^{2}+C\|(v,w)\|_{L^{2}}^{2}\|\nabla\Psi\|_{L^{2}}^{2}\nonumber\\
   &\leq \frac{1}{4}\|\Delta u\|_{L^{2}}^{2}+C(\|(\nabla v,\nabla w)\|_{L^{2}}^{2}+1).
\end{align}
Plugging \eqref{eq3.6} and \eqref{eq3.8} into \eqref{eq3.5}, we have
\begin{align}\label{eq3.9}
  \frac{d}{dt}\|\nabla u\|_{L^{2}}^{2}+\|\Delta u\|_{L^{2}}^{2}&\leq  C \|u\|_{L^{p}}^{q}\|\nabla u\|_{L^{2}}^{2}+C(\|(\nabla v,\nabla w)\|_{L^{2}}^{2}+1)\nonumber\\
  &\leq  C\left(\|u\|_{L^{p}}^{q}+\|(\nabla v,\nabla w)\|_{L^{2}}^{2}+1\right)\left(e+\|\nabla u\|_{L^{2}}^{2}\right).
\end{align}
Applying the Gronwall's inequality yields that
\begin{align*}
  \sup_{0\leq t\leq T}\|\nabla u(t)\|_{L^{2}}^{2}
  &\leq  \left(e+\|\nabla u_{0}\|_{L^{2}}^{2}\right)\exp\left\{\int_{0}^{T}\big(\|u(\tau)\|_{L^{p}}^{q}+\|(\nabla v(\tau),\nabla w(\tau))\|_{L^{2}}^{2}+1\big)d\tau\right\},
\end{align*}
which combining \eqref{eq2.6}, \eqref{eq2.19} and \eqref{eq3.9} give us to \eqref{eq3.4}.

On the other hand, plugging \eqref{eq3.7} and \eqref{eq3.8} into \eqref{eq3.5}, we have
\begin{align}\label{eq3.10}
  \frac{d}{dt}\|\nabla u\|_{L^{2}}^{2}+\|\Delta u\|_{L^{2}}^{2}
  &\leq  C\left(\|\nabla u\|_{L^{p}}^{\widetilde{q}}+\|(\nabla v,\nabla w)\|_{L^{2}}^{2}+1\right)\left(e+\|\nabla u\|_{L^{2}}^{2}\right).
\end{align}
Applying the Gronwall's inequality again yields that
\begin{align*}
  \sup_{0\leq t\leq T}\|\nabla u(t)\|_{L^{2}}^{2}
  &\leq  \left(e+\|\nabla u_{0}\|_{L^{2}}^{2}\right)\exp\left\{\int_{0}^{T}\big(\|\nabla u(\tau)\|_{L^{p}}^{\widetilde{q}}+\|(\nabla v(\tau),\nabla w(\tau))\|_{L^{2}}^{2}+1\big)d\tau\right\},
\end{align*}
which still leads to \eqref{eq3.4} by \eqref{eq2.6}, \eqref{eq2.19} and \eqref{eq3.10}. The proof of Lemma \ref{le3.1} is complete.
\end{proof}

Now we prove \eqref{eq3.1}.  Multiplying $\Delta v$ to the third equation of \eqref{eq1.1}, integrating over $\mathbb{R}^3$, we obtain
\begin{align}\label{eq3.11}
  \frac{1}{2}\frac{d}{dt}\|\nabla
  v\|_{L^{2}}^{2}+\|\Delta v\|_{L^{2}}^{2}&=-\int_{\mathbb{R}^{3}}(u\cdot\nabla)
  v\Delta v dx-\int_{\mathbb{R}^{3}}\nabla\cdot(v\nabla\Psi)\Delta v dx.
\end{align}
Applying Lemma \ref{le3.1}, the two terms in the right hand side of \eqref{eq3.10} can be estimated as follows:
\begin{align*}
  -\int_{\mathbb{R}^{3}}(u\cdot\nabla)
  v\Delta v dx&\leq \frac{1}{8}\|\Delta v\|_{L^{2}}^{2}+C\|(u\cdot\nabla) v\|_{L^{2}}^{2}\nonumber\\
  &\leq \frac{1}{8}\|\Delta v\|_{L^{2}}^{2}+C\|u\|_{L^{6}}^{2}\|\nabla v\|_{L^{3}}^{2}\nonumber\\
  &\leq \frac{1}{8}\|\Delta v\|_{L^{2}}^{2}+C\|\nabla u\|_{L^{2}}^{2}\|\nabla v\|_{L^{2}}\|\Delta v\|_{L^{2}}\nonumber\\
  &\leq \frac{1}{6}\|\Delta v\|_{L^{2}}^{2}+C\|\nabla v\|_{L^{2}}^{2},
\end{align*}
\begin{align*}
  -\int_{\mathbb{R}^{3}}\nabla\cdot(v\nabla\Psi)\Delta v dx
  &\leq \frac{1}{8}\|\Delta v\|_{L^{2}}^{2}+C\big(\|\nabla v\nabla\Psi\|_{L^{2}}^{2}+\|v\Delta\Psi\|_{L^{2}}^{2}\big)\nonumber\\
  &\leq \frac{1}{8}\|\Delta v\|_{L^{2}}^{2}+C\big(\|\nabla v\|_{L^{3}}^{2}\|\nabla \Psi\|_{L^{6}}^{2}+\|(v,w)\|_{L^{4}}^{4}\big)\nonumber\\
   &\leq \frac{1}{8}\|\Delta v\|_{L^{2}}^{2}+C\big(\|(v,w)\|_{L^{2}}^{2}\|\nabla v\|_{L^{2}}\|\Delta v\|_{L^{2}}+\|(v,w)\|_{L^{2}}^{\frac{5}{2}}\|(\Delta v,\Delta w)\|_{L^{2}}^{\frac{3}{2}}\big)\nonumber\\
  &\leq \frac{1}{6}\|(\Delta v, \Delta w)\|_{L^{2}}^{2}+C(1+\|(\nabla v, \nabla w)\|_{L^{2}}^{2}).
\end{align*}
Notice that same result also holds for $w$, thus we obtain
\begin{align}\label{eq3.12}
  \frac{d}{dt}(\|\nabla  v\|_{L^{2}}^{2}+\|\nabla  w\|_{L^{2}}^{2})+\frac{4}{3}(\|\Delta v\|_{L^{2}}^{2}+\|\Delta w\|_{L^{2}}^{2})\leq C(1+\|(\nabla v, \nabla w)\|_{L^{2}}^{2}).
\end{align}
Taking $\Delta$ on the first equation of \eqref{eq1.1}, multiplying the
resulting with $\Delta u$ and integrating over $\mathbb{R}^{3}$, by the condition
$\nabla\cdot u=0$,  we see that
\begin{align}\label{eq3.13}
  \frac{1}{2}\frac{d}{dt}\|\Delta
  u\|_{L^{2}}^{2}+\|\nabla\Delta u\|_{L^{2}}^{2}&=-\int_{\mathbb{R}^{3}}\Delta((u\cdot\nabla)
  u)\cdot\Delta u dx+\int_{\mathbb{R}^{3}}\Delta(\Delta\Psi\nabla\Psi)\cdot\Delta u dx.
\end{align}
Applying Lemma \ref{le3.1} again,  the right hand side of \eqref{eq3.13} can be bounded as follows:
\begin{align*}
  -\int_{\mathbb{R}^{3}}\Delta((u\cdot\nabla)
  u)\cdot\Delta u dx&\leq \frac{1}{8}\|\nabla \Delta u\|_{L^{2}}^{2} +C\|(\nabla u\cdot\nabla) u\|_{L^{2}}^{2}+\|(u\cdot\nabla)\nabla u\|_{L^{2}}^{2}\nonumber\\
  &\leq \frac{1}{8}\|\nabla \Delta u\|_{L^{2}}^{2} +C\|\nabla u\|_{L^{4}}^{4}+\|u\|_{L^{6}}^{2}\|\Delta u\|_{L^{3}}^{2}\nonumber\\
   &\leq \frac{1}{8}\|\nabla \Delta u\|_{L^{2}}^{2} +C\|\nabla u\|_{L^{2}}^{\frac{5}{2}}\|\nabla \Delta u\|_{L^{2}}^{\frac{3}{2}}+\|\nabla u\|_{L^{2}}^{2}\|\Delta u\|_{L^{2}}\|\nabla\Delta u\|_{L^{2}}\nonumber\\
   &\leq\frac{1}{4}\|\nabla \Delta u\|_{L^{2}}^{2} +C(\|\Delta u\|_{L^{2}}^{2}+1),
  \end{align*}
\begin{align*}
  \int_{\mathbb{R}^{3}}\Delta(\Delta\Psi\nabla\Psi)\cdot\Delta u dx&=\int_{\mathbb{R}^{3}}\nabla((v-w)\nabla\Psi)\cdot\nabla\Delta u dx\nonumber\\
  &\leq \frac{1}{4}\|\nabla\Delta u\|_{L^{2}}^{2}+C\Big(\|(\nabla v-\nabla w)\nabla\Psi\|_{L^{2}}^{2}+\|(v-w)\nabla^{2}\Psi\|_{L^{2}}^{2}\Big)\nonumber\\
  &\leq \frac{1}{4}\|\nabla\Delta u\|_{L^{2}}^{2}+C\Big(\|(\nabla v, \nabla w)\|_{L^{3}}^{2}\|\nabla\Psi\|_{L^{6}}^{2}+\|(v,w)\|_{L^{4}}^{4}\Big)\nonumber\\
  &\leq \frac{1}{4}\|\nabla\Delta u\|_{L^{2}}^{2}+ \frac{1}{6}\|(\Delta v, \Delta w)\|_{L^{2}}^{2}+C\Big(1+\|(\nabla v, \nabla w)\|_{L^{2}}^{2}\Big).
\end{align*}
Hence, we have
\begin{align}\label{eq3.14}
  \frac{d}{dt}\|\Delta
  u\|_{L^{2}}^{2}+\|\nabla\Delta u\|_{L^{2}}^{2}\leq \frac{1}{3}\|(\Delta v, \Delta w)\|_{L^{2}}^{2}+C\Big(1+\|(\nabla v, \nabla w)\|_{L^{2}}^{2}+\|\Delta
  u\|_{L^{2}}^{2}\Big).
\end{align}
Finally, we conclude from \eqref{eq3.12} and \eqref{eq3.14} that
\begin{align*}
  \frac{d}{dt}\big(\|\Delta
  u\|_{L^{2}}^{2}+\|(\nabla v, \nabla w)\|_{L^{2}}^{2}\big)&+\big(\|\nabla\Delta u\|_{L^{2}}^{2}+\|(\Delta v, \Delta w)\|_{L^{2}}^{2}\big)\nonumber\\
  &\leq C\big(1+\|(\nabla v, \nabla w)\|_{L^{2}}^{2}+\|\Delta
  u\|_{L^{2}}^{2}\big),
\end{align*}
which yields directly \eqref{eq3.1} by the Gronwall's inequality.
We complete the proof of  Theorem \ref{th1.3}.

\section{Proof of Theorem \ref{th1.4}}
In this section we prove Theorem \ref{th1.4}.  Based on the proofs of Theorems \ref{th1.2} and \ref{th1.3}, it is sufficiently to show that if there exists a  positive constant $K$ such that
\begin{align*}
  \int_{0}^{T_{*}}\|\nabla_{h} u^{h}(\cdot,t)\|_{\dot{B}^{0}_{p,\frac{2p}{3}}}^{q}dt\leq K\ \ \text{with}
  \ \ \ \frac{2}{q}+\frac{3}{p}=2\ \ \ \text{and}\ \ \ \frac{3}{2}<p\leq\infty,
\end{align*}
then \eqref{eq3.4} still holds, which gives us the desired estimate \eqref{eq3.1}. Actually, we shall establish  $L^{2}$ bounds of the vorticity $\omega=\nabla\times u$, which leads to the desired bounds of $\nabla u$ by the boundedness property of Riesz operators in $L^{2}$.
Taking $\nabla\times$  on the first equation
of \eqref{eq1.1}, we see that
\begin{equation}\label{eq4.1}
  \partial_{t}\omega+(u\cdot\nabla)\omega-\Delta \omega=(\omega\cdot\nabla)u+\nabla\times(\Delta \Psi\nabla \Psi).
\end{equation}
Multiplying \eqref{eq4.1} by $\omega$, and integrating over $\mathbb{R}^{3}$, one has
\begin{align}\label{eq4.2}
  \frac{1}{2}\frac{d}{dt}\|\omega\|_{L^{2}}^{2}+\|\nabla \omega\|_{L^{2}}^{2}&=\int_{\mathbb{R}^{3}}(\omega\cdot\nabla)u\cdot \omega dx-\int_{\mathbb{R}^{3}}(\Delta\Psi\nabla \Psi)\cdot(\nabla\times \omega)dx\nonumber\\
  &\stackrel{def}{=}K_{1}+K_{2}.
\end{align}
Notice that
\begin{align*}
  K_{1}=\sum_{i,j=1}^{3}\int_{\mathbb{R}^{3}}\omega_{i}\partial_{i}u^{j}\omega_{j} dx
  &=\sum_{i,j=1}^{2}\int_{\mathbb{R}^{3}}\omega_{i}\partial_{i}u^{j}\omega_{j} dx
  +\sum_{i=1}^{2}\int_{\mathbb{R}^{3}}\omega_{i}\partial_{i}u^{3}\omega_{3} dx\\
  &+\sum_{j=1}^{2}\int_{\mathbb{R}^{3}}\omega_{3}\partial_{3}u^{j}\omega_{j} dx
  +\int_{\mathbb{R}^{3}}\omega_{3}\partial_{3}u^{3}\omega_{3} dx.
\end{align*}
Thus, by using the facts
\begin{align*}
  \partial_{3}u^{3}=-(\partial_{1}u^{1}+\partial_{2}u^{2}), \ \ \ \omega_{3}=\partial_{1}u^{2}-\partial_{2}u^{1},
\end{align*}
we can easily see that
\begin{equation}\label{eq4.3}
  K_{1}\leq \int_{\mathbb{R}^{3}}|\nabla_{h}u^{h}||\omega|^{2}dx.
\end{equation}
Applying the Littlewood-Paley dyadic decomposition \eqref{eq1.13} for $\nabla_{h}u^{h}$, i.e.,
\begin{equation*}
  \nabla_{h}u^{h}=\sum_{j<-N}\Delta_{j}\nabla_{h}u^{h}+\sum_{j=-N}^{N}\Delta_{j}\nabla_{h}u^{h}+\sum_{j>N}\Delta_{j}\nabla_{h}u^{h},
\end{equation*}
we can split $K_{1}$ into
\begin{align}\label{eq4.4}
  K_{1}&\leq \sum_{j<-N}\int_{\mathbb{R}^{3}}|\Delta_{j}\nabla_{h}u^{h}||\omega|^{2}dx+\sum_{j=-N}^{N}\int_{\mathbb{R}^{3}}|\Delta_{j}\nabla_{h}u^{h}||\omega|^{2}dx
  +\sum_{j>N}\int_{\mathbb{R}^{3}}|\Delta_{j}\nabla_{h}u^{h}||\omega|^{2}dx\nonumber\\
  &\stackrel{def}{=}K_{11}+K_{12}+K_{13}.
\end{align}
Now we estimate the terms $K_{1j}$ ($j=1,2,3$) one by one. For $K_{11}$, by using the H\"{o}lder inequality, the Bernstein inequality \eqref{eq1.14} and the fact $\|\nabla u\|_{L^{2}}\leq C\|\omega\|_{L^{2}}$, one has
\begin{align}\label{eq4.5}
  K_{11}&=\sum_{j<-N}\int_{\mathbb{R}^{3}}|\Delta_{j}\nabla_{h}u^{h}||\omega|^{2}dx\nonumber\\
  &\leq C\sum_{j<-N}\|\Delta_{j}\nabla_{h}u^{h}\|_{L^{\infty}}\|\omega\|_{L^{2}}^{2}\nonumber\\
  &\leq C \sum_{j<-N}2^{\frac{3}{2}j}\|\Delta_{j}\nabla_{h}u^{h}\|_{L^{2}}\|\omega\|_{L^{2}}^{2}\nonumber\\
  &\leq C \Big(\sum_{j<-N}2^{3j}\Big)^{\frac{1}{2}}\Big(\sum_{j<-N}\|\Delta_{j}\nabla_{h}u^{h}\|_{L^{2}}^{2}\Big)^{\frac{1}{2}}\|\omega\|_{L^{2}}^{2}\nonumber\\
  &\leq C2^{-\frac{3}{2}N}\|\nabla u\|_{L^{2}}\|\omega\|_{L^{2}}^{2}\nonumber\\
  &\leq C2^{-\frac{3}{2}N}\|\omega\|_{L^{2}}^{3}.
\end{align}
For $K_{12}$, let $p'$ be a conjugate index of $p$, i.e., $\frac{1}{p}+\frac{1}{p'}=1$. It follows from the H\"{o}lder inequality that
\begin{align}\label{eq4.6}
  K_{12}&=\sum_{j=-N}^{N}\int_{\mathbb{R}^{3}}|\Delta_{j}\nabla_{h}u^{h}||\omega|^{2}dx\nonumber\\
  &\leq C\sum_{j=-N}^{N}\|\Delta_{j}\nabla_{h}u^{h}\|_{L^{p}}\|\omega\|_{L^{2p'}}^{2}\nonumber\\
  &\leq C\Big(\sum_{j=-N}^{N}1^{\frac{2p}{2p-3}}\Big)^{\frac{2p-3}{2p}}
  \Big(\sum_{j=-N}^{N}\|\Delta_{j}\nabla_{h}u^{h}\|_{L^{p}}^{\frac{2p}{3}}\Big)^{\frac{3}{2p}}\|\omega\|_{L^{2p'}}^{2}\nonumber\\
  &\leq CN^{\frac{2p-3}{2p}}\|\nabla_{h}u^{h}\|_{\dot{B}^{0}_{p,\frac{2p}{3}}}\|\omega\|_{L^{2p'}}^{2}\nonumber\\
  &\leq CN^{\frac{2p-3}{2p}}\|\nabla_{h}u^{h}\|_{\dot{B}^{0}_{p,\frac{2p}{3}}}\|\omega\|_{L^{2}}^{2(1-\frac{3}{2p})}\|\nabla\omega\|_{L^{2}}^{\frac{3}{p}}\nonumber\\
  &\leq \frac{1}{8}\|\nabla\omega\|_{L^{2}}^{2}+CN\|\nabla_{h}u^{h}\|_{\dot{B}^{0}_{p,\frac{2p}{3}}}^{ q}\|\omega\|_{L^{2}}^{2},
\end{align}
where $q=\frac{2p}{2p-3}$, and we have used the
following Gagliardo-Nirenberg inequality:
\begin{equation*}
  \|f\|_{L^{\frac{2p}{p-1}}}\leq \|f\|_{L^{2}}^{1-\frac{3}{2p}}\|
  \nabla f\|_{L^{2}}^{\frac{3}{2p}}\ \ \text{for}\ \ p\geq\frac{3}{2}.
\end{equation*}
For $K_{13}$, notice that
\begin{align*}
  \|\nabla^{2} u\|_{L^{2}}=\|\Delta u\|_{L^{2}}\leq C\|\nabla \omega\|_{L^{2}}.
\end{align*}
Therefore, based on the above inequality, applying the H\"{o}lder inequality,  the Bernstein inequalities \eqref{eq1.14} and \eqref{eq1.15}
again, $K_{13}$ can be estimated as
\begin{align}\label{eq4.7}
 K_{13}&=\sum_{j>N}\int_{\mathbb{R}^{3}}|\Delta_{j}\nabla_{h}u^{h}||\omega|^{2}dx\nonumber\\
  &\leq C\sum_{j>N}\|\Delta_{j}\nabla_{h}u^{h}\|_{L^{3}}\|\omega\|_{L^{2}}\|w\|_{L^{6}}\nonumber\\
  &\leq C \sum_{j>N}2^{\frac{j}{2}}\|\Delta_{j}\nabla_{h}u^{h}\|_{L^{2}}\|\omega\|_{L^{2}}\|\nabla \omega\|_{L^{2}}\nonumber\\
  &\leq C \Big(\sum_{j>N}2^{-j}\Big)^{\frac{1}{2}}\Big(\sum_{j>N}2^{2j}\|\Delta_{j}\nabla_{h}u^{h}\|_{L^{2}}^{2}\Big)^{\frac{1}{2}}\|\omega\|_{L^{2}}\|\nabla \omega\|_{L^{2}}\nonumber\\
  &\leq C2^{-\frac{N}{2}}\|\omega\|_{L^{2}}\|\nabla^{2} u\|_{L^{2}}\|\nabla \omega\|_{L^{2}}\nonumber\\
  &\leq C2^{-\frac{N}{2}}\|\omega\|_{L^{2}}\|\nabla \omega\|_{L^{2}}^{2}.
\end{align}
Plugging estimates \eqref{eq4.5}--\eqref{eq4.7} into
\eqref{eq3.4} gives us to
\begin{align}\label{eq4.8}
  K_{1}\leq \frac{1}{8}\|\nabla\omega\|_{L^{2}}^{2}+ C2^{-\frac{3}{2}N}\|\omega\|_{L^{2}}^{3}
  +CN\|\nabla_{h}u^{h}\|_{\dot{B}^{0}_{p,\frac{2p}{3}}}^{q}\|\omega\|_{L^{2}}^{2}
  +C2^{-\frac{N}{2}}\|\omega\|_{L^{2}}\|\nabla \omega\|_{L^{2}}^{2}.
\end{align}
Notice that  $K_{2}$ has been treated as $I_{2}$ in \eqref{eq2.23},
thus inserting estimates \eqref{eq4.8} and \eqref{eq2.23} into
\eqref{eq4.2}, we obtain
\begin{align}\label{eq4.9}
  \frac{d}{dt}\|\omega \|_{L^{2}}^{2}+\frac{3}{2}\|\nabla \omega\|_{L^{2}}^{2}&\leq C2^{-\frac{3}{2}N}\|\omega\|_{L^{2}}^{3}
  +CN\|\nabla_{h}u^{h}\|_{\dot{B}^{0}_{p,\frac{2p}{3}}}^{q}\|\omega\|_{L^{2}}^{2}
  +C2^{-\frac{N}{2}}\|\omega\|_{L^{2}}\|\nabla \omega\|_{L^{2}}^{2}\nonumber\\
  &
  +C(\|(\nabla v,\nabla w)\|_{L^{2}}^{2}+1).
\end{align}
Now if we choose $N$ sufficiently large such that
\begin{equation*}
   C2^{-\frac{N}{2}}\|\omega\|_{L^{2}}\leq \frac{1}{2},
\end{equation*}
i.e.,
\begin{equation*}
   N\geq \frac{\ln(\|\omega\|_{L^{2}}^{2}+e)+2\ln C}{\ln2}+2,
\end{equation*}
then the inequality \eqref{eq4.9} yields that
\begin{align}\label{eq4.10}
  \frac{d}{dt}\|\omega \|_{L^{2}}^{2}+\|\nabla \omega\|_{L^{2}}^{2}
  \leq
  C\big(1+\|\nabla_{h}u^{h}\|_{\dot{B}^{0}_{p,\frac{2p}{3}}}^{q}+\|(\nabla v,\nabla w)\|_{L^{2}}^{2}\big)
  \left(\|\omega\|_{L^{2}}^{2}+e\right)\ln\left(\|\omega\|_{L^{2}}^{2}+e\right).
\end{align}
Setting $Z(t)\stackrel{def}{=}e+\|\omega\|_{L^{2}}^{2}$. It follows from \eqref{eq4.10} that
\begin{align*}
  \frac{d}{dt}\ln Z(t)\leq
  C\big(1+\|\nabla_{h}u^{h}\|_{\dot{B}^{0}_{p,\frac{2p}{3}}}^{q}+\|(\nabla v,\nabla w)\|_{L^{2}}^{2}\big)\ln Z(t).
\end{align*}
Integrating on the time interval $[0,t]$, then taking maximum norm with respect to $t$ for any $0\leq t\leq T_{*}$, we conclude that
\begin{align*}
  \sup_{0\leq t\leq T_{*}}Z(t)\leq Z(0)\exp\Big\{\exp\big\{C\int_{0}^{T_{*}}\big(1+\|\nabla_{h}u^{h}(\tau)\|_{\dot{B}^{0}_{p,\frac{2p}{3}}}^{q}+\|(\nabla v(\tau),\nabla w(\tau))\|_{L^{2}}^{2}\big)d\tau\big\}\Big\}.
\end{align*}
Namely,
\begin{align*}
  \sup_{0\leq t\leq T_{*}}&\|\omega(t)\|_{L^{2}}^{2}\leq (e+\|\omega_{0}\|_{L^{2}}^{2})\nonumber\\
  &\times\exp\Big\{\exp\big\{CT_{*}+C\int_{0}^{T_{*}}(\|\nabla_{h}u^{h}(\tau)\|_{\dot{B}^{0}_{p,\frac{2p}{3}}}^{q}+\|(\nabla v(\tau),\nabla w(\tau))\|_{L^{2}}^{2})d\tau\big\}\Big\},
\end{align*}
 which gives us to $u\in L^{\infty}(0,T_{*}; H^{1}(\mathbb{R}^{3}))\cap L^{2}(0,T_{*}; H^{2}(\mathbb{R}^{3}))$. We complete the proof of Theorem
\ref{th1.4} .

\medskip
\medskip

\noindent\textbf{Acknowledgments.} J. Zhao is partially supported by the National Natural Science Foundation of China (11501453), the Fundamental Research Funds for the Central Universities (2014YB031) and the Fundamental Research Project of Natural Science in Shaanxi Province--Young Talent Project (2015JQ1004). M. Bai is partially supported by the Foundation for Distinguished Young Talents in Higher Education of Guangdong (2014KQNCX223). 

\end{document}